\documentclass[a4paper,12pt]{amsart}
\usepackage{color}
\usepackage{dsfont}
\usepackage{amsfonts,latexsym,amssymb,amscd,amsxtra}
\usepackage{ifthen}
\usepackage{graphicx,bm}

\newtheorem{thm}{Theorem}[section]
\newtheorem{lem}{Lemma}[section]
\newtheorem{cor}{Corollary}
\theoremstyle{definition}
\newtheorem{definition}[thm]{Definition}

\newtheorem{prop}{Proposition}[section]
\newtheorem{remark}[thm]{Remark}

\newtheorem{problem}{Problem}
\newtheorem{theorem}{Theorem}
\newtheorem{lemma}{Lemma}

\newtheorem{exe}{Exercise}
\newtheorem{exa}{Example}

\newtheorem{question}{Question}

\newtheorem{conjecture}{Conjecture}

\newcommand{\blem}{\begin{lemma}}
	\newcommand{\elem}{\end{lemma}}
\newcommand{\bexer}{\begin{exe}}
	\newcommand{\eexer}{\end{exe}}
\newcommand{\beq}{\begin{eqnarray}}
	\newcommand{\eeq}{\end{eqnarray}}
\newcommand{\bthm}{\begin{theorem}}
	\newcommand{\ethm}{\end{theorem}}
\newcommand{\beg}{\begin{exa}}
	\newcommand{\eeg}{\end{exa}}

\newcommand{\bdefe}{\begin{definition}}
	\newcommand{\edefe}{\end{definition}}

\newcommand{\bprop}{\begin{prop}}
	\newcommand{\eprop}{\end{prop}}

\newcommand{\bpf}{\begin{proof}}
	\newcommand{\epf}{\end{proof}}

\def\be{\begin{equation}}
	\def\ee{\end{equation}}
\newtheorem{cl}{Claim}
\newcommand{\bcl}{\begin{cl}}
	\newcommand{\ecl}{\end{cl}}

\newenvironment{pf}[1][]{%
	\vskip 3mm
	\noindent
	\ifthenelse{\equal{#1}{}}%
	{{\slshape Proof. }}%
	{{\slshape #1.} }%
}%
{\qed\bigskip}

	\newcommand{\zb}{{\overline{z}}}

		\newcommand{\T}{\mathbb{U}}

	\newcommand{\bquess}{\begin{questions}}
		\newcommand{\equess}{\end{questions}}
	\newcommand{\br}{\begin{remark}}
		\newcommand{\er}{\end{remark}}
	\newcommand{\brs}{\begin{remarks}}
		\newcommand{\ers}{\end{remarks}}
	
	\newcommand{\bn}{\begin{nonsec}}
		\newcommand{\en}{\end{nonsec}}

	\newcommand{\ds}{\displaystyle}

	\newcommand{\bcor}{\begin{cor}}
		\newcommand{\ecor}{\end{cor}}
	\newcommand{\bprob}{\begin{problem}}
		\newcommand{\eprob}{\end{problem}}
	\newcommand{\bcon}{\begin{conjecture}}
		\newcommand{\econ}{\end{conjecture}}
	\newcommand{\bques}{\begin{question}}
		\newcommand{\eques}{\end{question}}

	\newcommand{\begs}{\begin{examples}}
		\newcommand{\eegs}{\end{examples}}

	\newcommand{\bdefes}{\begin{definitions}}
		\newcommand{\edefes}{\end{definitions}}

	\newcommand{\ba}{\begin{array}}
		\newcommand{\ea}{\end{array}}

	\newcommand{\beqq}{\begin{eqnarray*}}
		\newcommand{\eeqq}{\end{eqnarray*}}
	\newcommand{\bee}{\begin{enumerate}}
		\newcommand{\eee}{\end{enumerate}}

	\newcommand{\bei}{\begin{itemize}}
		\newcommand{\eei}{\end{itemize}}
	\newcommand{\bed}{\begin{description}}
		\newcommand{\eed}{\end{description}}
	
	\newcommand{\bo}{\begin{obser}}
		\newcommand{\eo}{\end{obser}}
	\newcommand{\bos}{\begin{obsers}}
		\newcommand{\eos}{\end{obsers}}

	\newcommand{\ID}{{\mathbb D}}

	\numberwithin{equation}{section}

	\newcommand{\D}{\mathbb{D}}

\newcounter{alphabet}

\newenvironment{Thm}[1][]{\refstepcounter{alphabet}%
	\bigskip%
	\noindent%
	{\bf Theorem \Alph{alphabet}}%
	\ifthenelse{\equal{#1}{}}{}{ (#1)}%
	{\bf .} \itshape}{\vskip 8pt}

\begin{document}

\bibliographystyle{amsplain}
\title{$H^p$-Norm estimates of the partial derivatives and Schwarz lemma for $\alpha$-harmonic functions}

\def\thefootnote{}
\footnotetext{ \texttt{\tiny  \number\day-\number\month-\number\year
          }
} \makeatletter\def\thefootnote{\@arabic\c@footnote}\makeatother

\newcommand{\ben}{\begin{enumerate}}
	\newcommand{\een}{\end{enumerate}}


\author{Adel Khalfallah}
\address{Department of Mathematics, King Fahd University of Petroleum and
	Minerals, Dhahran 31261, Saudi Arabia}\email{khelifa@kfupm.edu.sa}

\author{ Miodrag Mateljevi\'c}
\address{M. Mateljevi\'c, Faculty of mathematics, University of Belgrade, Studentski Trg 16, Belgrade, Republic of Serbia}
\email{miodrag@matf.bg.ac.rs}
\begin{abstract}
Suppose $\alpha>-1$ and $1\leq p \leq \infty$. Let $f=P_{\alpha}[F]$ be an $\alpha$-harmonic mapping on $\D$ with the boundary $F$ being absolute continuous  and $\dot{F}\in L^p(0,2\pi)$, where $\dot{F}(e^{i\theta}):=\frac{dF(e^{i\theta})}{d\theta}$. In this paper, we investigate the membership of $f_z$ and $f_{\overline{z}}$ in the space $\mathcal{H}_{\mathcal{G}}^{p}(\mathbb{D})$, the generalized Hardy space. We prove, if $\alpha>0$, then both $f_z$ and $f_{\overline{z}}$ are in $\mathcal{H}_{\mathcal{G}}^{p}(\mathbb{D})$. If $\alpha<0$, then $f_z$ and $f_{\overline{z}}\in \mathcal{H}_{\mathcal{G}}^{p}(\mathbb{D})$ if and only if $f$ is analytic. Finally, we investigate a Schwartz Lemma for $\alpha$-harmonic functions.
\end{abstract}

\maketitle \pagestyle{myheadings}
\markboth{ A. Khalfallah  and M Mateljevi\'c}{}

\section{Preliminaries}

We denote by $\D$ 
  the unit disk and   $\T:=\partial\D$  the
unit circle.
For $z=x+iy\in\mathbb{C}$, the two complex
differential operators are defined by

$$u_z=\partial_z(u)=\frac{1}{2}\left(u_ x-iu_ y\right)~\mbox{and}~
u_{\bar z}=\overline{\partial}_z(u)=\frac{1}{2}\left( u_x+iu_y\right),$$
where $u$ is a complex valued function on $\D$.

\subsection{ $\alpha$-harmonic functions}\hfill

 For $\alpha > -1$, we denote by $\Delta _{\alpha}$ the weighted Laplace operator corresponding to the so-called standard weight $w_{\alpha}(z)=(1-|z|^{2})^{\alpha}$, that is, for $z\in \mathbb{D}$, \beq\label{eq0.1}
\Delta_{\alpha,z}=\partial_{z} w_{\alpha} (z)^{-1} \overline{\partial}_{z}. 
\eeq

The weighted Laplacians of the form \eqref{eq0.1} were first systematically studied by 
 Garabedian in \cite{Ga}.

Let $g\in\mathcal{C}(\mathbb{D})$
and $f\in\mathcal{C}^{2}(\mathbb{D})$. Of particular interest to us is the following {\it inhomogeneous $\alpha$-harmonic equation} in $\mathbb{D}$:
\beq\label{eq1.1}
\Delta_{\alpha}(f)=g.
\eeq
We also consider the associated \emph{Dirichlet boundary value problem} of functions $f$, satisfying the equation \eqref{eq1.1},
\beq\label{eq1.2}
\left\{
\begin{aligned}
\Delta_{\alpha}(f)&=g\;\;\;\;\text{in}\;\mathbb{D},\\
                 f&=f^{\ast}\;\;\text{on}\;\mathbb{T}.
\end{aligned}
\right.
\eeq
Here the boundary data $f^{\ast}$ is a distribution on $\mathbb{T}$, i.e. $f^{\ast}\in \mathcal{D}'(\mathbb{T})$,  and the boundary condition in \eqref{eq1.2} is  understood as $f_{r}\to f^{\ast}\in \mathcal{D}'(\mathbb{T})$ as $r \to 1^{-}$, where

$$
f_{r}(e^{i\theta}):=f(re^{i\theta}),\quad r\in[0,1).
$$

If $g\equiv0$ in \eqref{eq1.1}, the solutions to \eqref{eq1.1} are said to be {\it$\alpha$-harmonic functions}. Obviously, $0$-harmonic functions are simply  harmonic functions. 

In \cite{oA}, Olofsson and Wittsten showed that if an $\alpha$-harmonic function $f$ satisfies
$$\lim_{r\to 1^{-}}f_{r}=f^{\ast} \in \mathcal{D}'(\mathbb{T}) \;\; (\alpha > -1),$$ then it has the form of a \emph{Poisson type integral}
\beq\label{eq1.4}
f(z)=\mathcal{P}_{\alpha}[f^{\ast}](z)=\frac{1}{2\pi}\int^{2\pi}_{0}P_{\alpha}(ze^{-i\theta})f^{\ast}(e^{i\theta})d\theta
\eeq in $\mathbb{D}$,
where
$$P_{\alpha}(z)=\frac{(1-|z|^{2})^{\alpha+1}}{(1-z)(1-\overline{z})^{\alpha+1}}$$
is the {\it $\alpha$-harmonic Poisson kernel} in $\ID$.

Moreover, 
$$
P_\alpha(z)=\sum_{k=-\infty}^{\infty} e_{\alpha,k}(z), \quad z\in \D,
$$
where 
$$ e_{\alpha,k}(z)=z^k, \quad k\geq 0,$$ and 
$$ e_{\alpha,-k}(z)=\frac{1}{B(k,\alpha+1)}  \left( \int_0^1 t^{k-1} (1-t|z|^2)^{\alpha}dt\right) \overline{z}^{k}, \quad k\geq 1, $$ see \cite{oA}. 

In particular, we obtain

$$e_{\alpha,k}(z)= \mathcal{P}_\alpha[e^{ik\theta}](z), \quad  k\in \mathbb{Z}.  $$

In addition, we have
\be\label{1.5}
\overline{\partial} e_{\alpha,-k}(z)= \frac{1}{B(k,\alpha+1)} (1-|z|^2)^\alpha\,  \zb^{k-1}, \quad k\geq 1,
\ee
where $B$ is the beta function, see \cite[Lemma 1.1]{oA}.\\

\noindent Recall that  $$M_{\alpha}(r)=\frac{1}{2\pi}\int_{\T} |P_{\alpha}(ze^{-it})|dt.$$
 The mapping $M_\alpha$ is increasing on $[0,1)$ and
 \be\label{calpha}
 \lim_{r\to 1}M_{\alpha}(r)=c_\alpha=\frac{\Gamma(\alpha+1)}{\Gamma^2(\alpha/2+1)}.
\ee

\noindent See \cite{LWX,olo} for related discussions.\\

\subsection{Generalized Hardy spaces}\hfill

For $p\in(0,\infty]$, the {\it generalized Hardy space
$\mathcal{H}^{p}_{\mathcal{G}}(\mathbb{D})$} consists of all
measurable  functions from $\mathbb{D}$ to $\mathbb{C}$ such that
$M_{p}(r,f)$ exists for all $r\in(0,1)$, and $ \|f\|_{p}<\infty$,
where
$$M_{p}(r,f)=\left(\frac{1}{2\pi}\int_{0}^{2\pi}|f(re^{i\theta})|^{p}\,d\theta\right)^{\frac{1}{p}}
$$
and
$$\|f\|_{p}=
\begin{cases}
\displaystyle\sup \{M_{p}(r,f):\; 0<r <1\}
& \mbox{if } p\in(0,\infty),\\
\displaystyle\sup\{|f(z)|:\; z\in\mathbb{D}\} &\mbox{if } p=\infty.
\end{cases}$$

The classical {\it Hardy space $\mathcal{H}^{p}(\mathbb{D})$}  (resp. $h^{p}(\mathbb{D})$) is the set of  all  elements of $\mathcal{H}^{p}_{\mathcal{G}}(\mathbb{D})$ which are analytic (resp. harmonic) on $\D$. (cf. \cite{Du,Du1}).
Obviously, $\mathcal{H}^{p}_{\mathcal{G}}(\mathbb{D})\subset L^p(\mathbb{D})$ for each $p\in(0,\infty]$.
We denote by  $h^p_\alpha(\D)$  the corresponding Hardy space for $\alpha$-harmonic functions.\\

Denote by $L^{p}(\mathbb{T})~(p\in[1,\infty])$ the space of all measurable functions  $F$
of $\mathbb{T}$ into $\mathbb{C}$ with

$$\|F\|_p=
\begin{cases}
\displaystyle\left(\frac{1}{2\pi}\int_{0}^{2\pi}|F(e^{i\theta})|^{p}d\theta\right)^{\frac{1}{p}}
& \mbox{if } p\in[1,\infty),\\
\displaystyle \sup\{|F(e^{i\theta})|:\; \theta\in [0, 2\pi)\} &\mbox{if } p=\infty.\\
\end{cases}
$$

It is well-known that if $F$ is absolutely continuous, then it is of bounded variation. This implies that for
almost all $e^{i\theta}\in\mathbb{T}$, the derivative $\dot{F}(e^{i\theta})$ exists, where
$$\dot{F}(e^{i\theta}):=\frac{dF(e^{i\theta})}{d\theta}.$$

In this paper, we consider the following problem :\\

For $\alpha>-1$, under what conditions on the boundary function $F$ ensure that the
partial derivatives of its $\alpha$-harmonic extension $f=\mathcal{P}_\alpha[F]$, i.e., $f_z$ and $f_{\overline{z}}$, are in the space $\mathcal{H}_{\mathcal{G}}^{p}(\mathbb{D})$ (or $L^{p}(\mathbb{D})$), where $p\in  [1,\infty]$?\\

It is worth noting that a similar problem for harmonic functions  was treated in \cite{Zhu} and improved in \cite{CPW,KM}. 

\begin{Thm}{\rm (\cite[Theorem 1.2]{Zhu}),\cite[Theorem 1.1]{CPW})}\label{Zhu-1}
Suppose that $f=P[F]$ is a harmonic mapping in $\mathbb{D}$ and
$\dot{F}\in L^{p}(\mathbb{T})$, where $F$ is an  absolutely continuous  function. \ben
\item
If $p\in[1,\infty)$, then both $f_{z}$ and $\overline{f_{\overline{z}}}$ are in $L^{p}(\mathbb{D}).$
\item
If $p=\infty$, then there exists a harmonic
mapping $f=P[F]$, with  $\dot{F}\in L^{\infty}(\mathbb{T})$, such that
neither $f_{z}$ nor $\overline{f_{\overline{z}}}$ is in $L^{\infty}(\mathbb{D}).$
\een
\end{Thm}

Furthermore, under some additional conditions of $f$, they proved that the  partial derivatives are in $H^p(\D)$, for $p\in[1,\infty]$.

\begin{Thm}{\rm (\cite[Theorem 1.3]{Zhu},\cite[Theorem 1.2]{CPW})}\label{Zhu-2}
Suppose that $p\in[1,\infty]$ and $f=P[F]$ is a $(K,K')$-elliptic mapping in $\mathbb{D}$ with
$\dot{F}\in L^{p}(\mathbb{T})$, where $F$ is an absolute continuous function, $K\geq1$ and $K'\geq0$. Then both
$f_{z}$ and $\overline{f_{\overline{z}}}$ are in $H^{p}(\mathbb{D}).$
\end{Thm}

In \cite{CPW}, the authors showed  that Theorem B also holds true for harmonic elliptic mappings, which are more general than harmonic quasiregular mapping.\\

In \cite{KM},  we provide  a refinement of the two previous theorems, we prove that for $1<p<\infty$, both
$f_{z}$ and $\overline{f_{\overline{z}}}$ are in $H^{p}(\mathbb{D})$ without any extra conditions on 
$f$.

\begin{Thm}\cite{KM}
Suppose that $F$ is an  absolute continuous  function on $\T$ and  $f=P[F]$ is a harmonic mapping in $\mathbb{D}$ and
$\dot{F}\in L^{p}(\mathbb{T})$.
\ben
\item
If $p\in(1,\infty)$, then both $f_{z}$ and $\overline{f_{\overline{z}}}$ are in $\mathcal{H}^{p}(\mathbb{D}).$ Moreover, there exists a constant $A_p$ such that
$$ \max(\|f_z \|_p, \|\overline{f_{\overline{z}}} \|_p) \leq A_p \| \dot{F}\|_p .$$

\item If $p=1$ or $p=\infty$, then both  
  $f_{z} $ and  $\overline{f_{\overline{z}}}$ are in $\mathcal{H}^p(\mathbb{D})$  if and only if  $H(\dot{F}) \in L^p(\T)$. Moreover,
  $$2izf_z(z)=P[\dot{F}+iH(\dot{F})](z),$$
\een
where $H(\dot{F})$ denotes the Hilbert transformation of $\dot{F}$.
\end{Thm}

\section{Main results}

Assume $z=re^{i\theta}\in \D$, the polar partial derivative of $f$ with respect to $\theta$ is given by :

\be
\partial_\theta f  =i(zf_z-\bar{z}f_{\bar{z}}).
\ee

\subsection{Estimate of the angular derivative $\partial_{\theta}f$, $f$ is $\alpha$-harmonic.}

 \begin{thm}\label{thm2.1}
 	Suppose $\alpha>-1$ and $1\leq p \leq \infty$. Let $f=P_{\alpha}[F]$ be an $\alpha$-harmonic mapping on $\D$ with the boundary $F$ being absolute continuous  with $\dot{F}\in L^p(0,2\pi)$. Then
 	$$\partial_\theta f= \mathcal{P}_\alpha[{\dot{F}}]  \in h_\alpha^p(\D)$$
 	and
 	$$\|\partial_\theta  f \|_{p} \leq c_\alpha \|\dot{F}\|_{p},$$
where $c_\alpha$ is defined in (\ref{calpha}). 	
  \end{thm}
\begin{pf}
Assume that $f$ is $\alpha$-harmonic function and
 $$f=\mathcal{P}_{\alpha}[F],$$ with $F$ is absolute continuous and $\dot{F}\in L^p(0,2\pi)$, with $p\geq 1$.
 $$f(z)=\mathcal{P}_\alpha [F](z)=\frac{1}{2\pi} \int_\T  P_\alpha(ze^{-it})F(e^{it})\, dt.$$
 For $z=re^{i\theta}$, we have
 $$\partial_{\theta} P_\alpha (ze^{-it})= -\partial_{t} P_\alpha (ze^{-it}).$$
By integration by parts, we deduce
\be\label{eq:1}
\partial_\theta f= \mathcal{P}_\alpha[{\dot{F}}].
\ee

\noindent Using Jensen's inequality, we have
 $$|\partial_\theta f(re^{i\theta})|^p\leq M^{p-1}_\alpha(r) \frac{1}{2\pi}\int_0^{2\pi} |P_\alpha(ze^{-it})| |\dot{F}(e^{it})|^p \, dt.$$

\noindent Using Fubini theorem, we obtain
 $$\frac{1}{2\pi} \int_{0}^{2\pi}  |\partial_\theta f(re^{i\theta})|^p d\theta \leq M^p_{\alpha}(r) \|\dot{F}\|_{L^p}^p.$$
 Thus,
 $$\sup_{0<r<1} \frac{1}{2\pi} \int_{0}^{2\pi}  |\partial_\theta f (re^{i\theta})|^p d\theta \leq c^p_{\alpha} \|\dot{F}\|_{L^p}^p.$$

 For $p=\infty$, it is clear that
 $$|\partial_\theta f(re^{i\theta})| \leq M_\alpha(r) \|\dot{F}\|_\infty \leq c_\alpha \|\dot{F}\|_\infty.$$
Therefore, we obtain the desired result.
 \end{pf}

 \begin{lem}
Let $\alpha>-1$ and  $f=\mathcal{P}_{\alpha}[F]$ be an $\alpha$-harmonic mapping on $\D$ with the boundary function $F$ being absolute continuous. Then 
 \be
\zb \overline{\partial}f(z)= -   \frac{1}{2\pi i}\int_{0}^{2\pi} \frac{(1-|z|^2)^\alpha}{(1-\zb e^{it})^{\alpha+1}} \,  \dot{F}(e^{it})\, dt.
\ee
 \end{lem}
 
 \begin{pf}
 Let $\alpha>-1$, an easy computation shows that 

$$\overline{\partial} P_\alpha(z)=\frac{\alpha+1}{(1-\zb)^2} \left(\frac{1-|z|^2}{1-\zb}\right)^\alpha=(\alpha+1)\frac{(1-|z|^2)^\alpha}{(1-\zb)^{\alpha+2}}.$$
 As 
$f(z)=\frac{1}{2\pi} \int_0^{2\pi} P_\alpha(ze^{-it}) F(e^{it})\, dt,$ we obtain 

$$\overline{\partial}f(z)=\frac{1}{2\pi} \int_0^{2\pi}  \overline{\partial}[P_\alpha(ze^{-it})] F(e^{it})\, dt=\frac{1}{2\pi} \int_0^{2\pi}  \overline{\partial}P_\alpha (ze^{-it}) e^{it} F(e^{it})\,dt. $$
Hence,

$$ \overline{\partial} f(z)=(\alpha+1)(1-|z|^2)^\alpha\frac{1}{2\pi} \int_0^{2\pi}\frac{e^{it}}{(1-\zb e^{it})^{\alpha+2}}\, F(e^{it})\, dt. $$

Using integration by parts, we obtain
\be\label{eq:2}
 \overline{\partial} f(z)= -\frac{1}{\zb} \frac{1}{2\pi i}\int_{0}^{2\pi} \frac{(1-|z|^2)^\alpha}{(1-\zb e^{it})^{\alpha+1}} \dot{F}(e^{it})\, dt.
\ee

 \end{pf}
 
 Denote
$$ I_\alpha(r)=\frac{1}{2\pi} \int_0^{2\pi} \frac{(1-|z|^2)^\alpha}{|1-\zb e^{it}|^{\alpha+1}} dt,$$
where $r=|z|.$

\begin{lem}\cite[Proposition 1.1]{KMM}\label{lem2:2}
	Let $r\in[0,1)$.
	\begin{enumerate}
		\item If $\ds\alpha >0$, then 	
		\begin{equation*}
			I_{\alpha} (r) \leq\frac{\Gamma(\alpha)}{\Gamma^2(\alpha/2+1/2)}.
		\end{equation*}
		\item If $\ds -1<\alpha<0$, then
		\begin{equation*}
			I_{\alpha} (r) \leq\frac{\Gamma(-\alpha)}{\Gamma^2(1/2-\alpha/2)}
			(1-r^2)^\alpha.
		\end{equation*}
	\end{enumerate}
	
\end{lem}
 
 \begin{prop}\label{prop2:2}	Suppose $\alpha>-1$ and $1\leq p \leq \infty$. Let $f=\mathcal{P}_\alpha[F]$ be an $\alpha$-harmonic mapping on $\D$ with  $F$ is absolute continuous and satisfies $\dot{F}\in L^p(0,2\pi)$. Then
 
     $$M_p(r,\zb \overline{\partial}f(z))  \leq I_\alpha(|z|) \|\dot{F}\|_{p}.$$
 \end{prop}

\noindent  We  use the equation (\ref{eq:2}) and apply  Jensen's inequality to estimate the mean of the function  $\zb \overline{\partial}f(z)$.
 
\begin{proof}
 Using (\ref{eq:2}), we have

\be\label{eq:4}
|\zb \overline{\partial}f(z)| \leq \frac{1}{2\pi} \int_0^{2\pi} \frac{(1-|z|^2)^\alpha}{|1-\zb e^{it}|^{\alpha+1}} \,  |\dot{F}(e^{it})|\, dt.
\ee

Using Jensen's inequality, we get

$$|\zb \overline{\partial}f(z)|^p \leq I_\alpha^{p-1}(r)  \frac{1}{2\pi}\int_0^{2\pi} \frac{(1-|z|^2)^\alpha}{|1-\zb e^{it}|^{\alpha+1}} \,  |\dot{F}(e^{it})|^p\, dt.$$

Applying  Fubini's theorem, we obtain

\be \label{eq 5:5}
M_p^p(r,\zb \overline{\partial}f(z))  \leq I_\alpha^p(|z|) \|\dot{F}\|^{p}_{p}.
\ee
\end{proof}

\begin{cor}\label{prop2:3}
     If $\alpha>0$ and $f=\mathcal{P}_\alpha [F]$ and $F$ is absolute continuous such that $\dot{F}\in L^p$ with $1\leq p\leq \infty$. Then  $\zb \overline{\partial}f\in H_g^p(\D)$ and 
$$ \| \zb \overline{\partial}f(z)\|_p \leq  \frac{\Gamma(\alpha)}{\Gamma^2(\alpha/2+1/2)} \|\dot{F}\|_p.$$
\end{cor}
\begin{proof}
Let $\alpha>0$, by Proposition \ref{prop2:2}  and Lemma \ref{lem2:2}, we deduce that 
$$M_p(r,\zb \overline{\partial}f(z))   \leq I_\alpha(r) \|\dot{F}\|_p\leq \frac{\Gamma(\alpha)}{\Gamma^2(\alpha/2+1/2)} \|\dot{F}\|_{p}.$$
Therefore
$$ \| \zb \overline{\partial}f(z)\|_p \leq  \frac{\Gamma(\alpha)}{\Gamma^2(\alpha/2+1/2)} \|\dot{F}\|_{p}.$$
\end{proof}
Combining the previous results with the identity
\be\label{eq:2.8}
iz\partial =\partial_\theta+i\zb\bar{\partial},
\ee
 we obtain our first main result 
\begin{thm}
Let $\alpha \in (-1,\infty)$ with $\alpha\not=0$ and let $f=P_\alpha[F]$ and $F$ is absolute continuous  such that $\dot{F}\in L^p$ with $1\leq p\leq \infty$.
\begin{enumerate}
    \item  If $\alpha>0$, then
 $\overline{\partial}f$ and $\partial f$ are in  $H_{\mathcal{G}}^p(\D)\subset L^p(\D)$.
 \item 
If $\alpha\in (-1,0)$, then $\partial f$ and $\bar{\partial} f$ are in $L^p(\D)$ for $p<- \frac{1}{\alpha}$.

\item For $\alpha \in (-1,0)$ and  $p\geq -\frac{1}{\alpha}$ there exits $f$ an  $\alpha$-harmonic function such that $\partial f$   and $\bar{\partial} f$ $\not\in L^p(\D)$, moreover,  $\partial f$   and $\bar{\partial} f$ $\not\in H_{\mathcal{G}}^1(\D)$.
\end{enumerate}
  \end{thm}

\begin{proof}
  (1) and (2) are direct consequences of the previous discussion and  Proposition \ref{prop2:2}   and Lemma \ref{lem2:2}.\\
  (3) Let  $\alpha \in (-1,0)$ and  consider $H(e^{i\theta})=e^{-i\theta} $. We have $e_{\alpha,-1}(z)=P_\alpha[H](z)$ and 
  $$\overline{\partial} e_{\alpha,-1}(z)=(\alpha+1) (1-|z|^2)^\alpha.$$
  Hence $\overline{\partial} e_{\alpha,-1} \not \in L^p(\D)$ for $p\geq -\frac{1}{\alpha}$ and $\overline{\partial} e_{\alpha,-1} \not \in H^1_{\mathcal{G}}(\D)$.
\end{proof}

Our second main result is the following 

\begin{thm}
Let $\alpha \in (-1,0)$ and let $f=\mathcal{P}_\alpha[F]$ and $F$ is absolute continuous  such that $\dot{F}\in L^p$ with $1\leq p\leq \infty$.
Then 
    $\overline{\partial}f$ or $\partial f$ is in  $H_{\mathcal{G}}^p(\D)$, if and only if,   $f$ is analytic. 
\end{thm}

\begin{proof}
  Consider the series expansion  of $f$
  
  $$
  f(z)= \sum_{n=1}^\infty a_{-n} e_{\alpha,-n}(z)+ \sum_{n=0}^\infty a_{n} z^n.
  $$

For $n\geq 1$, let
$$f_n(z):= \frac{1}{2\pi}\int_0^{2\pi} f(ze^{it}) e^{int}dt, \quad z\in \D.$$

It yields
$$
f_n(z)=a_{-n} e_{\alpha,-n}(z).
$$

Hence
$$
\bar{\partial}f_n(z)= a_{-n}\bar{\partial}e_{\alpha,-n}(z)= \frac{1}{2\pi} \int_0^{2\pi} \bar{\partial f}(ze^{it})e^{i(n-1)t}
dt.$$

First, assume that $\overline{\partial} f \in H_{\mathcal{G}}^p(\D)$, then by (\ref{1.5}), it yields
$
a_{-n}(1-|z|^2)^{\alpha} 
$ is bounded, this implies that $a_{-n}=0$ for $n \geq 1$ and $f$ is analytic. 

In the case $\partial{f}\in H_{\mathcal{G}}^p(\D)$, we deduce that  $\overline{\partial} f \in H_{\mathcal{G}}^p(\D)$. Indeed, by   (\ref{eq:2.8}) and Theorem \ref{thm2.1}, we obtain
$$i\zb \overline{\partial} f(z) =i z\partial f(z) - \partial_\theta f(z) \in  H_{\mathcal{G}}^p(\D).$$

 Conversely,  assume that $f$ is analytic, then $f=P_0[F]$
 and  $f'(z)=-\frac{i}{z} \partial_\theta f (z)=-\frac{i}{z}P[\dot{F}]\in H^p(\D).$
 \end{proof}

\section{Schwartz Lemma for $\alpha$-harmonic functions}

The Schwarz lemma for analytic functions plays a vital role in complex analysis and has been generalized to various spaces of functions.

Heinz \cite{Heinz}  generalized it to the class of complex-valued harmonic functions, that
is, if $f$ is a complex-valued harmonic function from $\D$ into itself with $f(0) = 0$, then for $z \in \D$,
$$
\ f(z)|\leq \frac{4}{\pi}
\arctan|z|.
$$
Hethcote \cite{Heth} improved Heinz’s result, by removing the assumption $f(0) = 0$, i .e., let $f$ be a harmonic function from $\D$ to $\D$, then

$$
\left| f(z)-\frac{1-|z|^2}{1+|z|^2} f(0) \right| \leq  \frac{4}{\pi}
\arctan|z|.
$$

In \cite{Pli}, Li et al.  stated a Schwarz type lemma for solutions of the $\alpha$-harmonic equation $\Delta_\alpha(f)=g$, with the condition $\alpha \geq 0$. For $g=0$, they obtain

\begin{Thm}\cite[Theorem 2.1]{Pli}
Let $\alpha \geq 0$ and $f^*\in \mathcal{C}(\T)$ and $f=P_\alpha[f^*]$  an $\alpha$-harmonic function on $\D$ such that $f(0)=0$. Then 
$$
|f(z)|\leq 2^\alpha \frac{4}{\pi} |f^*|_\infty \arctan |z|.
$$
\end{Thm}

In the proof, the authors used the assumption $P_\alpha[|f^*|](0)=0$ instead of $P_\alpha[f^*](0)=0$. Clearly the assumption  $P_\alpha[|f^*|](0)=0$ implies that $f^*=0$ and thus $f=0$. In order to fill this gap, in \cite{LC}, the authors proved the following 

\begin{Thm}\cite[Theorem 1.1]{LC} Suppose that $f^*\in \mathcal{C}(\T)$ and $f=P_\alpha[f^*]$. 
\begin{enumerate}
    \item[(a)] If $\alpha \geq 0$, then
    $$
    |f(z)|\leq 2^{\alpha+1} \frac{|f|_\infty}{\pi} \arctan \left(\frac{1+|z|}{1-|z|} \tan \frac{c\pi}{2} \right).
    $$
      \item[(b)] If $\alpha <0 $, then
    $$
    |f(z)|\leq  2^{1-\alpha} \frac{|f|_\infty}{\pi} (1-|z|^2)^\alpha \arctan \left(\frac{1+|z|}{1-|z|} \tan \frac{c\pi}{2} \right),
    $$
\end{enumerate}
where  $\displaystyle c=\frac{P_\alpha[|f^*|](0)}{|f^*|_\infty}$.
\end{Thm}

In our opinion, it is more natural to provide Schwarz type lemmas for $\alpha$-harmonic functions involving $f(0)$  instead of  $P_\alpha[|f^*|](0)$.\\

The aim of this section is to prove a Schwarz type lemma for $\alpha$-harmonic functions  from the unit disc to itself in the spirit of Hethcote \cite{Heth}.\\

Let $$ g_\alpha(z):= \left(\frac{1-|z|^2}{1-\overline{z}} \right)^\alpha ,$$
so
 \be
P_\alpha(z)=\left( \frac{1-|z|^2}{1-\overline{z}} \right)^\alpha \frac{1-|z|^2}{|1-z|^2}= g_\alpha(z)  P(z),
\ee
where
$$P(z)=\frac{1-|z|^2}{|1-z|^2},$$
is the classical Poisson kernel.\\

Let 
$$|g_\alpha(r)|_\infty =\sup_{0\leq \theta\leq 2\pi} |g_\alpha(re^{i\theta})|.$$
One can check that 
\be \label{gr1}
|g_\alpha(r)|_\infty \leq 2^\alpha, \quad \mbox{for} \quad  \alpha \geq 0,
\ee
and 
\be\label{gr2}
|g_\alpha(r)|_\infty \leq (1-r)^\alpha, \quad \mbox{for}\quad  \alpha<0.
\ee
Let $f:\D \to \D$ be an $\alpha$-harmonic mapping from the unit disc to itself. Then, we can write
$$f(z)=\frac{1}{2\pi}\int_\T P_\alpha(ze^{-it}) f^*(e^{it})\, dt, $$
where $f^*$ is the boundary function of $f$, with $|f^*|_\infty \leq 1.$

Let
$$G_\alpha(z):=g_\alpha[f] (z)= \frac{1}{2\pi}\int_\T g_\alpha(ze^{-it}) f^*(e^{it})\, dt= \frac{(1-|z|^2)^\alpha}{2\pi}  \int_\T \frac{ f^*(e^{it})}{(1-\zb e^{it})^\alpha}\, dt.$$

Remark, in the case $\alpha=0$, $G_0(z)=f(0)$.

\begin{lem} Let $\alpha> -1$ and $f$ be an $\alpha$-harmonic function from the unit disc to itself. Then

$$\left| f(z)-\frac{1-|z|^2}{1+|z|^2}G_\alpha(z)\right| \leq \frac{4}{\pi}|g_\alpha(|z|)|_\infty \arctan|z|.$$
In particular 
\begin{enumerate}
    \item If $\alpha\geq 0$, then 
	$$\left| f(z)-\frac{1-|z|^2}{1+|z|^2}G_\alpha(z)\right| \leq \frac{2^{\alpha+2}}{\pi} \arctan|z|.$$
	\item If $\alpha< 0$, then 
	$$\left| f(z)-\frac{1-|z|^2}{1+|z|^2}G_\alpha(z)\right| \leq \frac{4}{\pi} (1-|z|)^\alpha\arctan|z|.$$
\end{enumerate}
\end{lem}

\begin{proof}
We have

$$\left| f(z)-\frac{1-|z|^2}{1+|z|^2}G_\alpha(z)\right| \leq \frac{1}{2\pi} \int_\T  \left| P(ze^{-it})-\frac{1-|z|^2}{1+|z|^2}  \right| \left|g_\alpha (ze^{-it}) f^*(e^{it})\ \right|\, dt  $$

$$ \leq \frac{4}{\pi} |f^*|_\infty |g_\alpha(|z|)_\infty \arctan |z|.$$

Indeed, by Hethcote \cite{Heth}, it yields  $\displaystyle \frac{1}{2\pi}\int_\T \left| P(ze^{-it})-\frac{1-|z|^2}{1+|z|^2} \right| dt \leq \frac{4}{\pi}\arctan |z|. $
The conclusion follows from the inequalities (\ref{gr1}) and (\ref{gr2}).
\end{proof}

\begin{lem}Let $\alpha>-1$ and $r\in [0,1)$. Then
	$$\frac{1}{2\pi} \int_\T \left|  (1-r e^{it})^{-\alpha}-1 \right|\, dt  \leq  \left| (1-r)^{-\alpha}-1 \right|.$$
\end{lem}

\begin{pf}
Set $u(r,t)=(1-r e^{it})^{-\alpha}.$
Since  $ u_r(r,t)=  \alpha e^{it} (1-r
e^{it})^{-\alpha-1}$, and $-\alpha-1<0$, we have  $$|u_r (r,t)|\leq |\alpha| ( 1- r)^{-\alpha-1},\quad r\in [0,1] \mbox{ and } t\in \mathbb{R}.$$
 Next,     using  $u(r,t)- u(0,t)= \int_0^r u_x (x,t) dx$, we obtain

$|u(r,t)- u(0,t)|\leq  \int_0^r |u'_x (x,t)| dx\leq  \int_0^r |\alpha |(1-
x)^{-\alpha-1} dx = \frac{|\alpha|}{\alpha}( (1-  r)^{-\alpha}-1)$    and therefore 
  $$\frac{1}{2\pi} \int_\T \left|  (1-r e^{it})^{-\alpha}-1 \right|\, dt  \leq \left| ( 1- r)^{-\alpha}-1\right|.$$
\end{pf}

Combining the previous two lemmas, we obtain the main result of this section.

\begin{thm}
	Let $\alpha >-1$ and $f:\D \to \D$ be an $\alpha$-harmonic function. Then
	\begin{enumerate}
	    \item If $\alpha\geq 0$, then
		$$\left| f(z)-\frac{(1-|z|^2)^{\alpha+1}}{1+|z|^2}f(0)\right| \leq  \frac{2^{\alpha+2}}{\pi} \arctan|z|+   2^{\alpha+1} (1-|z|) \left(1-(1-|z|)^\alpha \right).$$
		\item If $\alpha< 0$, then
		$$\left| f(z)-\frac{(1-|z|^2)^{\alpha+1}}{1+|z|^2}f(0)\right| \leq  \frac{4}{\pi} (1-|z|)^\alpha\arctan|z|+ \left((1-|z|)^\alpha-1 \right).$$
\end{enumerate}

\end{thm}

\begin{proof}
Let
$$H_\alpha(z):=\frac{G_\alpha(z)}{(1-|z|^2)^\alpha}=\frac{1}{2\pi}  \int_\T \frac{ f^*(e^{it})}{(1-\zb e^{it})^\alpha}\, dt.$$
and

$$ K_\alpha(z):=H_\alpha(z)-f(0).$$

Notice that 
$$  f(z)-\frac{(1-|z|^2)^{\alpha+1}}{1+|z|^2}f(0)= \left(f(z)-\frac{1-|z|^2}{1+|z|^2}G_\alpha(z)\right)+ \frac{(1-|z|^2)^{\alpha+1}}{1+|z|^2}K_\alpha(z).$$
Using the previous lemmas, it yields:

(1) if $\alpha\geq 0$, then
$$\left| f(z)-\frac{(1-|z|^2)^{\alpha+1}}{1+|z|^2}f(0)\right| \leq \frac{2^{\alpha+2}}{\pi} |f^*|_\infty\arctan|z|+ \frac{(1-|z|^2)^{\alpha+1}}{1+|z|^2}|K_\alpha(z)|$$

$$K_\alpha(z)=H_\alpha(z)-f(0)=\frac{1}{2\pi} \int_\T \left[ (1-\zb e^{it})^{-\alpha}-1 \right]f^*(e^{it})\, dt. $$

Thus
$$ |K_\alpha(z)| \leq    \frac{1}{2\pi} \int_\T \left|  (1-|z| e^{it})^{-\alpha}-1 \right|\, dt \leq  (1-r)^{-\alpha} -1.$$

(2) First, we remark that 
$$\frac{(1-|z|^2)^{\alpha+1}}{1+|z|^2} \leq (1-|z|^2)^\alpha \leq  (1-|z|)^\alpha.$$
The case $\alpha<0$ is treated similarly as the previous case. 
\end{proof}

Using the identity for $p\in(0,1)$ and $x,y\geq 0$
$$ (x+y)^p \leq x^p +y^p$$ 
and the convexity of the function $(1-x)^\alpha$ , we obtain  $ (1-x)^\alpha \geq 1-\alpha x$ for $\alpha \geq 1$ and $x\in [0,1]$.  

\begin{cor} Let $f:\D \to \D$ be an $\alpha$-harmonic function.
	\begin{enumerate}
	    \item If $0 < \alpha \leq 1$, then
$$\left| f(z)-\frac{(1-|z|^2)^{\alpha+1}}{1+|z|^2}f(0)\right| \leq  2^{\alpha+1} \left(\frac{2}{\pi}\arctan |z|+  |z|^\alpha \right).$$

	\item If $\alpha\geq 1$, then
$$\left| f(z)-\frac{(1-|z|^2)^{\alpha+1}}{1+|z|^2}f(0)\right| \leq  2^{\alpha+1} \left(\frac{2}{\pi}\arctan |z|+ \alpha |z| \right).$$\\
\end{enumerate}
\end{cor}

We close the paper by providing the power series expansion of $H_\alpha$, as it is anti-analytic function.

Since
$$\frac{1}{(1-\zb e^{it})^\alpha}=\sum_{n=0}^\infty \frac{(\alpha)_n}{n!}
\, \zb^n e^{int},$$
we deduce that
$$ H_\alpha(z)=\frac{G_\alpha(z)}{(1-|z|^2)^\alpha}=f(0)+\sum_{n=1}^\infty \frac{(\alpha)_n}{n!} \widehat{f^*}(-n)\,\zb^n, $$
where
$$\widehat{f^*}(-n)=\frac{1}{2\pi}\int_\T f^*(e^{it}) e^{int}\, dt. $$

 Here we are using that $f$ has the following  series expansion  
  
  $$
  f(z)= \sum_{n=1}^\infty \widehat{f^*}(-n) e_{\alpha,-n}(z)+ \sum_{n=0}^\infty \widehat{f^*}(n) z^n,
  $$
  see \cite[Sectiuon 5]{oA}.


\end{document}